\documentclass{amsart}

\RequirePackage{setspace}
\usepackage{hyperref}
\RequirePackage{amsmath, amsfonts, amssymb, amsthm}
\RequirePackage[utf8]{inputenc}
\RequirePackage{mathrsfs}
\RequirePackage[english]{babel}
\RequirePackage[T1]{fontenc}
\RequirePackage{mathpple}
\RequirePackage{tikz-cd}
\RequirePackage{mathtools}
\RequirePackage{stmaryrd}
\RequirePackage{xcolor}
\RequirePackage{faktor}

\RequirePackage[%
left = \glqq,%
right = \grqq,%
leftsub = \glq,%
rightsub = \grq%
]{dirtytalk}

\newtheorem{thm}{Theorem}[section]

\newtheorem{lem}[thm]{Lemma}
\newtheorem{cor}[thm]{Corollary}

\newtheorem{defin}[thm]{Definition}
\newtheorem{rmk}[thm]{Remark}

\DeclareMathOperator{\id}{id}

\DeclareMathOperator{\dom}{dom}
\DeclareMathOperator{\ran}{ran}

\newcommand{\KKK}{\mathbb{K}}

\newcommand{\RRR}{\mathbb{R}}

\newcommand{\QQQ}{\mathbb{Q}}

\newcommand{\Ff}{\mathcal{F}}

\title[Function spaces on separable compact lines]{Function spaces on separable compact lines}

\begin{document}

\author[K.\ Kucharski]{Kacper Kucharski}
\address{Institute of Mathematics\\ University of Warsaw\\ ul. Banacha 2\\
02--097 Warszawa, Poland }
\email{k.kucharski6@uw.edu.pl}

\date{\today}
\begin{abstract}
In this paper, we provide a complete isomorphism classification of the spaces $C_p(K)$ of real-valued continuous functions endowed with the topology of pointwise convergence for separable compact lines $K$ of weight $\omega_1$, under the assumption of Baumgartner's Axiom $\mathsf{BA}$. More specifically, we show that, up to linear homeomorphism, there are exactly two function spaces $C_p(K)$ for such $K$. We also construct an example of a separable compact line $K$ of weight $2^{\omega}$ neither of whose spaces of continuous functions, $C_p(K)$ and $C_w(K)$ is homeomorphic to its square.
\end{abstract}

\thanks{The author was partially supported by the NCN (National Science Centre, Poland) research grant no.\ 2020/37/B/ST1/02613.}



\maketitle

\section{Introduction.}

The study of function spaces of the form $C(X)$, equipped with various natural topologies, is a central theme both in topology and functional analysis. A recurring question in this area is to what extent the structure of the underlying space $X$ determines the topological or linear--topological properties of the corresponding function space. In particular, classification problems for spaces $C(X)$ and $C_p(X)$, as well as questions concerning their factorization properties (e.g., whether they are homeomorphic or isomorphic to their squares) have attracted considerable attention (see \cite{A, M3, T2} and the references within).\smallskip

A {\it compact line} is any linearly ordered compact topological space. In their recent work, Korpalski, Koszmider and Marciszewski \cite[Theorem 1.3]{KKM} proved that, assuming Baumgartner's Axiom {\sf BA}, all Banach spaces $C(K)$, where $K$ is a separable compact line of weight $\omega_1$, are isomorphic. In particular, it is easy to see that
$C(K)$ is isomorphic to $C(K) \times C(K)$, for all such $K$ (see \cite[Corollary 1.4]{KKM}). On the other hand, under the assumption of {\sf CH}, there exists a family $\{K_{\alpha} \colon \alpha < 2^{2^{\omega}}\}$ of separable compact lines of weight $\omega_1$ such that the corresponding Banach spaces $C(K_{\alpha})$ are pairwise nonisomorphic \cite[Theorem 1.1]{KKM}.\smallskip

The aim of this paper is to investigate to what extent analogous results hold for the spaces $C_p(K)$, that is, the spaces of real--valued continuous functions equipped with the topology of pointwise convergence. It turns out that this situation somewhat differs from that of Banach spaces $C(K)$.\smallskip

First, we consider separable compact lines of weight $\omega_1$. Under the assumption of {\sf BA}, we obtain a complete classification of the spaces $C_p(K)$ for such $K$. More precisely, we show that there exist exactly two separable compact lines $L_0$ and $L_1$ such that for every separable compact line $K$ of weight $\omega_1$, there exists exactly one $i \in \{0,1\}$ such that the space $C_p(K)$ is isomorphic to $C_p(L_i)$ (see Theorem \ref{thm:main2}). As a corollary, we conclude that under {\sf BA}, for every separable compact line $K$ of weight $\omega_1$, the space $C_p(K)$ is isomorphic to its square $C_p(K) \times C_p(K)$. Note that under {\sf CH}, by the virtue of the Closed Graph Theorem, the number of pairwise nonisomorphic $C_p(K)$ spaces does not change from the Banach space topology of $C(K)$.
\smallskip

In the second part of the paper, we provide a {\sf ZFC} construction of a separable compact line $K$ of weight $2^{\omega}$ such that neither $C_p(K)$ nor $C_w(K)$ (that is, the space of real--valued continuous functions on $K$ endowed with the weak topology) is homeomorphic to its square (see Theorem \ref{thm:main}). As a corollary, we conclude that the corresponding Banach spaces $C(K)$ and $C(K) \times C(K)$ are nonisomorphic as well.
The construction is based on the technique of killing homeomorphisms, invented by Sierpiński \cite{S,vM2} and further developed by Marciszewski \cite{M1} in the context of function spaces.\smallskip

The question regarding the existence of a Banach space $E$ or a function space $C_p(X)$ not homeomorphic (or isomorphic) to its square has been studied and resolved. The first examples of an infinite compact space $K$ such that the space $C_p(K)$ is not homeomorphic to $C_p(K) \times C_p(K)$ were found simultaneously by Marciszewski \cite{M1} and Gul'ko \cite{G}. In fact, the example constructed by Marciszewski works both for the pointwise topology and for the weak topology.
Pol \cite{Po} found several examples of metrizable spaces $E$ with the space $C_p(E)$ not isomorphic to its square, whereas Krupski and Marciszewski \cite{KM} found an example of a metrizable space $E$ with the space $C_p(E)$ not even homeomorphic to its square. Our construction adds to this line of research by providing such an example in the class of separable compact lines, which, to our knowledge, has not been previously known.

\section{Notation.}

All spaces considered in this paper are assumed to be Tychonoff topological spaces. For a space $X$, by $C(X)$ we will denote the set of all real-valued continuous functions on $X$, and by $C_p(X)$ the space $C(X)$ endowed with the topology of pointwise convergence. If $X$ is compact, $C(X)$ equipped with the supremum norm is a Banach space. We use $C_w(X)$ to denote $C(X)$ endowed with the weak topology induced by its Banach space dual.\smallskip

As usual, $\omega$ denotes the first infinite ordinal, identified with the set of natural numbers $\{0,1,2,\dots\}$, and $\omega_1$ denotes the first uncountable ordinal. A set is called countable if it is either finite or countably infinite. For sets $X$ and $Y$, by $X \subset Y$ we mean that $X$ is contained in $Y$ or that $X$ is equal to $Y$.
\smallskip

For spaces $X$ and $Y$, we write $X \simeq Y$ whenever $X$ and $Y$ are homeomorphic. 
If the spaces $X$ and $Y$ additionally carry linear structures, then by an \emph{isomorphism} we mean a linear homeomorphism, and we write $X \sim Y$. 
In particular, this terminology is used both for Banach spaces of the form $C(K)$ and for spaces $C_w(K)$ and  $C_p(K)$. The symbol $X \oplus Y$ denotes the topological disjoint union of spaces $X$ and $Y$. By dimension, we always mean the topological covering dimension $\dim$ (see \cite{E2}).\smallskip

Throughout the paper, by $I$ we will denote the closed unit interval $[0,1]$. For $K \subset I$ and $A \subset K$, let $K_A$ be the following space
$$
K_A = (K \times \{0\}) \cup (A \times \{1\})
$$
equipped with the lexicographic order topology. The following theorem is a well--known characterization of separable compact lines (cf. \cite{Os}).

\begin{thm}[Ostaszewski]\label{thm:ostaszewski}
A space $X$ is a separable compact line if and only if $X$ is homeomorphic to $K_A$ for some closed $K \subset I$ and some $A \subset K$.
\end{thm}

The reader should note that if the set $A \subset K$ is infinite, then the weight of the space $K_A$ is just the cardinality $|A|$. It is also straightforward to see that the dimension of any separable compact line cannot exceed $1$.

\section{The linear--topological structure of $C_p(K)$ for separable compact lines $K$ of weight $\omega_1$ under {\sf BA}.}

Let us recall that a subset $A$ of a space $X$ is called $\omega_1$--dense if $|A \cap U| = \omega_1$ for any nonempty open $U \subset X$. Baumgartner's Axiom is the following statement:
\begin{center}
{\sf BA}: Every two $\omega_1$--dense subsets of reals are order isomorphic.
\end{center}
It was proved by Baumgartner (see \cite{B}) that {\sf BA} is consistent with {\sf ZFC}. Below, we will investigate the linear and topological structure of the spaces $C_p(K)$ for separable compact lines $K$ of weight $\omega_1$. In particular, we aim at providing the following isomorphic classification of such spaces, under the assumption of {\sf BA}.

\begin{thm}[\sf BA]\label{thm:main2}
Let $K$ be a separable compact line of weight $\omega_1$. Fix an $\omega_1$--dense set $B \subset (0,1)$ and let $L_0 = I_B$ and $L_1 = I_B \oplus I$. Then $C_p(K) \sim C_p(L_i)$ if and only if $\dim(K) = i$.
\end{thm}

As a consequence we will obtain that all such spaces are isomorphic to their square, which shows that the assumption of {\sf BA} is critical.
Indeed, below in Section \ref{sec:ex} we provide an example (see Theorem \ref{thm:main}) of a separable compact line $K$ of weight continuum such that the space $C_p(K)$ is not even homeomorphic to the space $C_p(K) \times C_p(K)$. This example, together with the aforementioned Theorem \ref{thm:main2}, will show that the sentence
\begin{center}
    {\it If $K$ is a separable compact line of weight $\omega_1$, then the space $C_p(K)$ is isomorphic to its square $C_p(K) \times C_p(K)$}
\end{center}
is independent of the axioms of set theory {\sf ZFC}.

\smallskip

Recall that for any spaces $X$ and $Y$, we have
$$
C_p(X) \times C_p(Y) \sim C_p(X \oplus Y).
$$
We will use this identification freely throughout the following arguments.
The reader should acknowledge the similarities in the proof below to the proof of the "Banach space version" of the result, given in \cite{KKM}.
Before tackling the proof of the result we will first provide a series of lemmas, from which the proof of Theorem \ref{thm:main2} will follow.\smallskip

One of the key ingredients for our argument is the characterization given by M\'{a}trai \cite{Ma} (see also \cite{Go}), of those spaces $X$ whose function space $C_p(X)$ is isomorphic to the space $C_p(I)$. Let us recall the notion of an {\it $I$--derivative}. 
\begin{defin}\label{def:I-der}
Let $X$ be a space and let
$$
I(X) = \bigcup\{U \colon U \; \text{is an open subset of $X$ which can be embedded into $I$}\}.
$$
Define $X^{[0]} = X$ and $X^{[n + 1]} = X^{[n]} \setminus I(X^{[n]})$ for all $n \in \omega$.
\end{defin}

\begin{thm}[M\'{a}trai]\label{thm:matrai}
For any space $X$ the following conditions are equivalent:
\begin{enumerate}
    \item $X$ is one dimensional, compact, metrizable and $X^{[m]} = \emptyset$ for some $m \in \omega$,
    \item $C_p(X) \sim C_p(I)$.
\end{enumerate}
\vspace{-20pt}\qed
\end{thm}

The corollary below is an immediate consequence of the above characterization.

\begin{cor}\label{lem:one_dim_compact}
If $K$ is a one--dimensional compact subset of $I$, then 
$$
C_p(K) \sim C_p(I).
$$
\end{cor}

\begin{cor}\label{cor:ctble_weight_double}
If $K$ is an infinite compact space, which can be embedded into the unit interval $I$, then
$$
C_p(K) \sim C_p(K) \times C_p(K).
$$
\end{cor}

\begin{proof}
We may assume that $K \subset I$. If $\dim K = 0$, then the conclusion instantly follows from the characterization provided by Baars and de Groot \cite[Theorem~2.13]{BdG}.
If $\dim K = 1$, then the result is achieved with the use of Corollary \ref{lem:one_dim_compact} and the well--established fact, that $C_p(I) \sim C_p(I) \times C_p(I)$ (which itself is also a straightforward corollary of Theorem \ref{thm:matrai}).
\end{proof}

The lemma below was proved in \cite[Proposition 4.6]{KKM}.

\begin{lem}\label{lem:scl_reduction}
If $L$ is a dense-in-itself separable compact line, then there exists $A \subset (0,1)$ such that $L$ is order isomorphic (and hence homeomorphic) to $I_A$.\qed
\end{lem}

For a space $X$ and a nonempty closed subspace $F \subset X$, let $C_{p,F}(X)$ denote the space of all $f \in C_p(X)$ vanishing on $F$, i.e., 
$$
C_{p,F}(X) = \{f \in C_p(X) \colon f|F \equiv 0\}.
$$
The next lemma should be compared with \cite[Lemma~4.3]{M2}.
\begin{lem}\label{lem:scl_linear_factor}
If $L$ is a separable compact line and $F$ is a nonempty closed subset of $L$ then
$$
C_p(L) \sim C_p(F) \times C_{p,F}(L).
$$
\end{lem}
\begin{proof}
It suffices (see \cite[Proposition 6.6.6]{vM}) to construct a continuous linear extender $e \colon C_p(F) \to C_p(L)$, i.e., a continuous linear map $e$ satisfying $e(f)|F = f$ for all $f \in C_p(F)$. By Theorem \ref{thm:ostaszewski}, we may assume without loss of generality that $L = K_A$ for some compact set $K \subset I$ and some $A \subset K$.\smallskip

Let $a = \inf(F)$ and $b = \sup(F)$. For $x \in [a,b] \setminus F$, where the interval $[a,b]$ is understood with respect to the lexicographic order $\prec$, define $a_x = (s_x,i_x) = \sup\{y \in F \colon y \prec x\}$ and $b_x = (t_x,j_x) = \inf\{y \in F \colon y \succ x\}$. Note that $s_x \prec t_x$ for all such $x$.
Finally, let us define $e \colon C_p(F) \to C_p(L)$ with
\begin{equation}
e(f)(x) =
\begin{cases}
    f(x), & \text{if } x \in F; \\
    f(b), & \text{if } x \succ b; \\
    f(a), & \text{if } x \prec a; \\
    \tfrac{t_x - t}{t_x - s_x}f(a_x) + \tfrac{t - s_x}{t_x - s_x}f(b_x), & \text{if } x \in [a,b] \setminus F;
\end{cases}
\end{equation}
for $f \in C_p(F)$ and $x \in L$. The reader readily checks that $e$ is the desired extender.
\end{proof}

The following lemma belongs to the mathematical folklore (see, e.g., \cite[Corollary~6.6.7]{vM}).

\begin{lem}\label{lem:factorize_R}
Let $X$ be a space containing a nontrivial convergent sequence. Then $C_p(X) \sim C_p(X) \times \RRR$.\qed
\end{lem}

Below, for a nonempty closed subset $F$ of a space $X$, by $C_{p,\infty}(X/F)$ we will mean the space of all continuous functions on the quotient space $X/F$ vanishing on the unique point $\infty$ of $\pi(F)$, where $\pi \colon X \to X/F$ is the corresponding quotient map.

\begin{cor}\label{cor:scl_factor_quotient}
If $L$ is an infinite separable compact line and $F$ is a nonempty closed subset of $L$, then
$$
C_p(L) \sim C_p(L/F) \times C_p(F)
$$
\end{cor}

\begin{proof}
First note that any infinite separable compact line contains a nontrivial convergent sequence. By Lemmas \ref{lem:factorize_R} and \ref{lem:scl_linear_factor} and the fact that 
$$
C_p(X) \sim \{f \in C_p(X) \colon f(x) = 0\} \times \RRR,
$$
for any space $X$ and point $x \in X$, we obtain:
\begin{align*}
    C_p(L) &\sim \RRR \times C_p(L) \sim \RRR \times C_{p,F}(L) \times C_p(F) \sim\\ 
    &\sim \RRR \times C_{p,\infty}(L/F) \times C_p(F) \sim C_p(L/F) \times C_p(F)
\end{align*}
which finishes the proof.
\end{proof}

\begin{lem}\label{lem:ctbl_compact_embedding}
Let $P$ be a countable compact space and let $A \subset (0,1)$. Then $P$ can be embedded into $I_A$.
\end{lem}

\begin{proof}
It suffices to observe that there is an embedding of $P$ into $I_0 = I \times \{0\} \subset I_A$. Since $P$ is a countable compact space, there exists an embedding $i \colon P \to \QQQ$.
Consider the space 
$$
Q = (\QQQ \cap (0,1)) \times \{0\}
$$
with the subspace topology from $I_A$, and note that since $Q$ is a countable, dense-in-itself, and first-countable space, it is homeomorphic to the rationals $\QQQ$. 
Now fix a homeomorphism $h \colon \QQQ \to Q$.
Then $h \circ i \colon P \to I_0$ is the desired embedding.
\end{proof}

\begin{lem}\label{lem:scl_dii_absorb_ctbl_compact}
If $K$ is a dense-in-itself separable compact line and $P$ is countable compact space, then
$$
C_p(K) \sim C_p(K) \times C_p(P).
$$
\end{lem}

\begin{proof}
Let $A \subset (0,1)$ be such that $K \simeq I_A$. Such a set $A$ exists by Lemma \ref{lem:scl_reduction}. By Lemma \ref{lem:ctbl_compact_embedding} we may assume that $P \subset I_A$. Now, if $P$ is infinite, then by Lemma \ref{lem:scl_linear_factor} and Corollary \ref{cor:ctble_weight_double} consecutively, we obtain:
\begin{align*}
C_p(I_A) &\sim C_{p,P}(I_A) \times C_p(P) \sim\\ 
&\sim C_{p,P}(I_A) \times C_p(P) \times C_p(P) \sim\\ 
&\sim C_p(I_A) \times C_p(P),
\end{align*}
which finishes the proof. If $P$ is finite, then it is discrete, and one easily obtains the proof with the aid of Lemma \ref{lem:factorize_R}.
\end{proof}

For an ordinal number $\alpha$, let $X^{(\alpha)}$ denote the $\alpha$'th Cantor--Bendixson derivative of the space $X$ (see \cite[p. 33]{Ke}).

\begin{lem}\label{lem:scl_CB_der}
Let $L$ be an uncountable separable compact line and let
$$
\alpha = \min \{\beta \colon L^{(\beta)} = L^{(\beta + 1)}\}.
$$
Then $C_p(L) \sim C_p(L^{(\alpha)})$.
\end{lem}

\begin{proof}
First, note that $|L \setminus L^{(\alpha)}| \leq \omega$ (see \cite[Corollary 4.5]{KKM}). Since $L^{(\alpha)}$ is a closed subspace of $L$, it is also an uncountable separable compact line. By Corollary \ref{cor:scl_factor_quotient} we have
\begin{equation}\label{eq:lem:CB1}
    C_p(L) \sim C_p(L/L^{(\alpha)}) \times C_p(L^{(\alpha)}).
\end{equation}
Observe that $L/L^{(\alpha)}$ is a countable compact space.
Since $L^{(\alpha)}$ is dense-in-itself, by Lemma \ref{lem:scl_dii_absorb_ctbl_compact} we have
\begin{equation}
    C_p(L/L^{(\alpha)}) \times C_p(L^{(\alpha)}) \sim C_p(L^{(\alpha)}),
\end{equation}
which by (\ref{eq:lem:CB1}) finishes the proof.
\end{proof}

\begin{lem}\label{lem:scl_reduction2}
Suppose $L$ is an uncountable separable compact line. Then there exists $B \subset (0,1)$ with $C_p(L) \sim C_p(I_B)$.
\end{lem}

\begin{proof}
By Lemma \ref{lem:scl_CB_der} we have $C_p(L) \sim C_p(L^{(\alpha)})$, where $\alpha = \min\{\beta \colon L^{(\beta)} = L^{(\beta + 1)}\}$. Since the space $L^{(\alpha)}$ is dense-in-itself, we may use Lemma \ref{lem:scl_reduction}, to obtain a set $B \subset (0,1)$ satisfying $C_p(L^{(\alpha)}) \sim C_p(I_B)$.
\end{proof}

\begin{lem}\label{lem:ctble_split}
If $A \subset (0,1)$ is a countable set, then $I_A$ can be embedded into $I$.
\end{lem}

\begin{proof}
First assume that the set $A$ is infinite and fix an injective enumeration $\{a_n \colon n \in \omega\}$ of $A$. Define a function $\iota \colon I_A \to I$ by the following formula:
\begin{equation}
    \iota(x,i) =
    \begin{cases*}
      \tfrac{1}{2}(x + \sum\{2^{-n-1} \colon n \in \omega \; \text{such that} \; a_n < x\}), & if $i = 0$; \\
      \tfrac{1}{2}(x + \sum\{2^{-n-1} \colon n \in \omega \; \text{such that} \; a_n \leq x\}), & if $i = 1$.
    \end{cases*}
  \end{equation}
It is straightforward to see that $\iota$ is an embedding. If the set $A$ is finite, one may produce an analogous formula.
\end{proof}

\begin{lem}\label{lem:open_split_factor}
Let $U$ be an open subset of the interval $(0,1)$ and let $A \subset (0,1)$. Then the quotient spaces
$$
Z = I_A / (I \setminus U)_{(A \setminus U)} \;\; \& \;\; Z' = I_{A \cap U} / \big((I \setminus U) \times \{0\}\big)
$$
are homeomorphic.
\end{lem}

\begin{proof}
Let $\pi_Z \colon I_A \to Z$ and $\pi_{Z'} \colon I_{A \cap U} \to Z'$ be appropriate quotient maps and let $\infty_Z$ and $\infty_{Z'}$ denote the unique points of $\pi_Z((I \setminus U)_{(A \setminus U)})$ and $\pi_{Z'}((I \setminus U \times \{0\}))$, respectively. Notice that both in $Z$ and in $Z'$, all points of the set $U_{A \cap U}$ are unglued, and that all the other points of $I_A$ and $I_{A \cap U}$ are identified by $\pi_Z$ and $\pi_{Z'}$ respectively. 
So we have shown that the map $h \colon Z \to Z'$ given by
$$
h(\infty_Z) = \infty_{Z'} \;\; \& \;\; h|U_{A \cap U} = \id_{U_{A \cap U}}
$$
is a bijection. Since the space $Z$ is compact, it suffices to check the continuity of $h$.
Since the domain $Z$ of the map $h$ is a quotient space, in order to see that $h$ is continuous, we need to check that the composition $h \circ \pi_Z \colon I_A \to Z'$ is a continuous map. Consider the map $p \colon I_A \to I_{A \cap U}$ given by
$$
p((x,i)) =
\begin{cases}
(x, i), & \text{if } x \in A \cap U \\
(x,0), & \text{if } x \notin A \cap U
\end{cases}
$$
and note that $p$ is continuous. Now fix an open set $V \subset Z'$.  Since $\pi_{Z'}$ is a quotient map, the set $\pi_{Z'}^{-1}(V)$ is open in $I_{A \cap U}$. Finally, we have
$$
(h \circ \pi_Z)^{-1}(V) = p^{-1}(\pi_{Z'}^{-1}(V)),
$$
which shows that the set $(h \circ \pi_Z)^{-1}(V)$ is open and hence the mapping $h$ is continuous, as desired.
\end{proof}

\begin{lem}\label{cor:f_sigma}
For any set $A \subset (0,1)$ and any open interval $J \subset (0,1)$ the subspace $J_{A \cap J}$ of $I_A$ is of type $F_{\sigma}$.
\end{lem}

\begin{proof}
First, notice that $J_{A \cap J}$ is an open subspace of $I_A$. Now, since the space $I_A$ is hereditarily Lindel\"{o}f, it also needs to be perfectly normal. In particular, all open subspaces of $I_A$ are of type $F_{\sigma}$ as desired.
\end{proof}

\begin{lem}\label{lem:almost}
If $K$ is a separable compact line of weight $\omega_1$, then there exist an $\omega_1$--dense set $B \subset (0,1)$ and a compact space $Z$, which can be embedded into $I$, such that %$\dim(Z) = \dim(K)$ and
$$
C_p(K) \sim C_p(I_B) \times C_p(Z).
$$
\end{lem}

\begin{proof}
Let $K$ be a separable compact line of weight $\omega_1$. By Lemma \ref{lem:scl_reduction2} there exists $A \subset (0,1)$ such that
\begin{equation}\label{eq:almost1}
    C_p(K) \sim C_p(I_A).
\end{equation}
Since $I_A$ must have weight equal to $\omega_1$ as well (see \cite[Corollary I.1.6]{A}), we have that $|A| = \omega_1$.\smallskip

Consider the following union of intervals:
$$
U = \bigcup \{(p,q) \colon p, q \in \QQQ \; \& \; |A \cap (p,q)| < \omega_1\}
$$
and note that $A \setminus U$ is $\omega_1$--dense in $I \setminus U$. Now consider the separable compact line $L = (I \setminus U)_{(A \setminus U)}$ and an ordinal
$$
\alpha = \min\{\beta \colon L^{(\beta)} = L^{(\beta + 1)}\}.
$$
By Lemma \ref{lem:scl_CB_der} we have
\begin{equation}\label{eq:almost2}
    C_p(L) \sim C_p(L^{(\alpha)}).
\end{equation}
Since $L^{(\alpha)}$ is dense-in-itself, by Lemma \ref{lem:scl_reduction}, there is $B \subset (0,1)$ such that $L^{(\alpha)} \simeq I_B$. Moreover, it is not difficult to see that $B$ must be $\omega_1$--dense in $(0,1)$ (see the end of the proof of Lemma 7.6 in \cite{KKM}). By Corollary \ref{cor:scl_factor_quotient} we have
\begin{equation}\label{eq:almost3}
    C_p(I_A) \sim C_p(L) \times C_p(I_A/L).
\end{equation}
It remains to show that there is a compact $Z \subset I$ such that
\begin{equation}\label{eq:almost4}
    C_p(I_A/L) \sim C_p(Z).
\end{equation}
For the rest of the argument, let us denote 
$$
H = I_A/L, \;\; T = I_{A \cap U} \;\; \& \;\;T' = I_{A \cap U}/((I \setminus U) \times \{0\}).
$$

% \begin{claim}
% The space $H$ is metrizable.
% \end{claim}
% To see that it is the case, first recall that if $X$ is a compact space, then $C_p(X)$ is separable if and only if $X$ has countable weight (see, e.g., \cite[Exercise~213]{T}). So it is enough to prove that the space $C_p(H)$ is separable. Now note that the space $T'$ has countable weight, because it is a continuous image of the space $T$, which itself has countable weight, since the set $A \cap U$ is countable. By Lemma \ref{lem:open_split_factor} we have $C_p(T') \sim C_p(H)$ which finishes the proof of the claim.

Note that $H$ is metrizable, as it is homeomorphic to $T'$ (by Lemma \ref{lem:open_split_factor}), which in turn is a continuous image of the compact metrizable space $T$. Next, we will show that there exists a compact subset $Z \subset I$ satisfying $C_p(H) \sim C_p(Z)$. If $H$ is zero-dimensional, it can be embedded into the Cantor set, and we can simply put $Z = H$. Assume, then, that $H$ is one-dimensional. In this case, it suffices to show that $C_p(H) \sim C_p(I)$. Let us take a closer look at the space $T'$.\smallskip

Let $\{I_n \colon n \in \omega\}$ be a family of pairwise disjoint open intervals in $I$ such that $U = \bigcup_{n \in \omega} I_n$ (if $U$ can be represented as a union of only finitely many pairwise disjoint open intervals, then the argument is analogous). Let $T_n$ denote the open subspace $(I_n)_{(A \cap I_n)}$ of $T$. Since $A \cap U$ is countable, by Lemma \ref{lem:ctble_split} used for $(J_n)_{(A \cap I_n)}$, where $J_n$ is the closure of $I_n$ in the euclidean topology,
every $T_n$ can be embedded into the unit interval $I$.\smallskip

Let $\pi \colon T \to T'$ denote the quotient map, let $\infty$ denote the unique point of $\pi((I \setminus U) \times \{0\})$ and let $T_n' = \pi(T_n)$ for each $n \in \omega$. It is easy to see that for each $n$ we have $T_n \simeq T_n'$, and so we may infer that the first $I$--derivative $(T')^{[1]}$ (see Definition \ref{def:I-der}) of the space $T'$ is either a single point $\{\infty\}$ or is empty. Then, we also have $(T')^{[2]} = \emptyset$, and so by Lemma \ref{lem:open_split_factor} and Theorem \ref{thm:matrai} respectively, we obtain
\begin{equation}\label{eq:almost5}
C_p(H) \sim C_p(T') \sim C_p(I).
\end{equation}
Now, by letting $Z = I$, we are done for the case of dimension $1$.\smallskip

Finally, note that $\dim H \leq 1$, by virtue of the Countable Sum Theorem for $\dim$ (see \cite[3.1.8]{E2}).
Indeed, note that for all $n$ we have $\dim T_n \leq 1$, and each $T_n$ is of type $F_{\sigma}$ (cf. Lemma \ref{cor:f_sigma}). The point $\infty$ is closed and so the claim holds true.
\smallskip

By putting observations (\ref{eq:almost1}) to (\ref{eq:almost5}) together, we obtain
$$
C_p(K) \sim C_p(I_B) \times C_p(Z),
$$
for $B$ and $Z$ as desired.
\end{proof}

\begin{lem}\label{lem:cantor_factor}
Let $B \subset (0,1)$ and let $Z$ be a zero--dimensional compact metrizable space. If $|B| < 2^{\omega}$, then
$$
C_p(I_B) \sim C_p(I_B) \times C_p(Z)
$$
\end{lem}

\begin{proof}
First note that if $Z$ is finite, then we obtain the conclusion with the aid of Lemma \ref{lem:factorize_R}. Now assume that $Z$ is infinite.
Since $2^{\omega} \simeq 2^{\omega} \times 2^{\omega}$, then $(0,1)$ contains continuum many pairwise disjoint copies of the Cantor set $2^{\omega}$, at least one of which being disjoint from the set $B$. Denote this copy of the Cantor set by $C$. Without loss of generality we may assume that $Z \subset C$. Notice that $Z' = Z \times \{0\}$ is a closed copy of $Z$ in $I_B$. Then, by Corollaries \ref{cor:scl_factor_quotient} and \ref{cor:ctble_weight_double} respectively it follows that
\begin{equation*}
\begin{split}
C_p(I_B) &\sim C_p(I_B/Z') \times C_p(Z') \sim\\
&\sim C_p(I_B/Z') \times C_p(Z') \times C_p(Z') \sim\\ 
&\sim C_p(I_B) \times C_p(Z)
\end{split}
\end{equation*}
as desired.
\end{proof}

We are now ready to provide a proof of Theorem \ref{thm:main2}.
 
\begin{proof}[Proof of Theorem \ref{thm:main2}]
First, notice that under {\sf BA} the spaces $I_B$ and $I_{B'}$ are homeomorphic, whenever $B,B'$ are $\omega_1$--dense subsets of $(0,1)$ (see \cite[Corollary 7.5]{KKM}). This clearly implies that $C_p(I_B) \sim C_p(I_{B'})$. One readily checks, that if $B \subset (0,1)$ is a dense set, then the space $I_B$ has a clopen basis, and hence $\dim(I_B) = 0$.\smallskip

Let $K$ be a separable compact line of weight $\omega_1$ and let $B$ be any $\omega_1$--dense subset of $(0,1)$. By Lemma~\ref{lem:almost} there exists an $\omega_1$--dense set $B' \subset (0,1)$ along with a compact set $Z \subset I$, such that
$$
C_p(K) \sim C_p(I_{B'}) \times C_p(Z).
$$
First assume that $Z$ is zero--dimensional. Under {\sf BA}, the continuum is strictly greater than $\omega_1$, so by Lemma \ref{lem:cantor_factor} we have
\begin{align*}
C_p(K) &\sim C_p(I_{B'}) \times C_p(Z) \sim C_p(I_{B'}) \sim C_p(I_{B}) = C_p(L_0).
\end{align*}
If $Z$ is one--dimensional, then by Corollary \ref{lem:one_dim_compact} we have $C_p(Z) \sim C_p(I)$ and so
\begin{align*}
C_p(K) &\sim C_p(I_{B'}) \times C_p(Z) \sim\\ 
&\sim C_p(I_{B}) \times C_p(I) \sim\\ 
&\sim C_p(I_{B} \oplus I) = C_p(L_1).
\end{align*}
By the theorem of Pestov (see \cite{P}), the spaces $C_p(L_0)$ and $C_p(L_1)$ cannot be isomorphic, since the spaces $L_0$ and $L_1$ have different dimensions, namely $\dim(L_i) = i$, for $i = 0,1$.
\end{proof}

\begin{cor}[\sf BA]
If $K$ and $L$ are separable compact lines of weight $\omega_1$, then $C_p(K) \sim C_p(L)$ if and only if $\dim(K) = \dim(L)$.
\end{cor}

\begin{proof}
Fix two separable compact lines $K$ and $L$ of weight $\omega_1$. If $\dim(K) \neq \dim(L)$, then the spaces $C_p(K)$ and $C_p(L)$ cannot be isomorphic, due to the aforementioned theorem by Pestov. If $\dim(K) = \dim(L)$, then the conclusion follows from Theorem \ref{thm:main2}.
\end{proof}

\begin{cor}[{\sf BA}]\label{cor:main2}
If $K$ is a separable compact line of weight $\omega_1$, then
$$
C_p(K) \sim C_p(K) \times C_p(K)
$$
\end{cor}

\begin{proof}
Under {\sf BA} the spaces $I_B$ and $I_B \oplus I_B$ are homeomorphic. It is also well known that the spaces $C_p(I)$ and $C_p(I \oplus I)$ are isomorphic (see, e.g., Theorem \ref{thm:matrai}), so the corollary follows with the use of Theorem \ref{thm:main2}.
\end{proof}

\begin{rmk}
Note that, by the Closed Graph Theorem, the corollary above implies that the Banach spaces $C(K)$ and $C(K) \times C(K)$ are isomorphic (see \cite[Corollary 1.4]{KKM}).
\end{rmk}

\section{An example.}\label{sec:ex}

Below we will provide an example of a separable compact line $K$ of weight $2^{\omega}$ such that the space of continuous functions $C(K)$ endowed with either weak or pointwise topology is not homeomorphic to its square.
\smallskip

Let $\KKK = I_{(0,1)}$. For $A \subset (0,1)$ let
$$
E(A) = \left\{f \in C(\KKK) \; \colon \; \forall x \in (0,1) \setminus A \quad f(x,0) = f(x,1)\right\}.
$$
It is not difficult to see that given $A \subset (0,1)$, the space $E(A)$ endowed with the weak topology (respectively, the pointwise topology) is isomorphic to $C_w(I_A)$ (respectively, to $C_p(I_A)$).

The construction will proceed according to the following scheme. First, we will show that there exist $2^{\omega}$ injective maps $\{\varphi_{\alpha} \colon \alpha < 2^{\omega}\}$ with the property that if, for some $A \subset (0,1)$, the map $\psi \colon E(A) \to E(A) \times E(A)$ is a homeomorphism, then there exists an $\alpha < 2^{\omega}$ such that $\psi = \varphi_{\alpha}|E(A)$. Then, by induction, we will construct a set $X \subset (0,1)$ of size $2^{\omega}$ as the union of a family of countable sets $\{X_{\alpha} \colon \alpha < 2^{\omega}\}$, chosen such that the set $X_{\alpha}$ prevents the restriction $\varphi_{\alpha}|E(X)$ from being a homeomorphism. Finally, the desired separable compact line $K$ is defined as $I_X$.
The construction below (see Theorem \ref{thm:main}) will be carried out for both the weak and the pointwise topologies simultaneously; however, the necessary lemmas will be proved explicitly only for the weak topology, since it seems to be more complicated. We will comment on how to obtain the analogous lemmas for the pointwise topology.\smallskip

For the arguments below, let us fix a countable dense set $\{q_n \colon n \in \omega\}$ in $\KKK$. Let $M(\KKK)$ denote the space of all Radon measures on $\KKK$, and let $\delta_x \in M(\KKK)$ denote the probability measure concentrated at $x$, for $x \in \KKK$. 
Observe that if a set $S \subset M(\KKK)$ contains the set of measures $\{\delta_{q_n} \colon n \in \omega\}$, then the mapping $i_S \colon C(\KKK) \to \RRR^S$ given by
$$
i_S(f)(\mu) = \mu(f)
$$
is injective. For $X \subset \KKK$, define $\Delta(X) = \{\delta_x \colon x \in X\}$. Additionally, for sets $S \subset T$ and a subspace $E \subset \RRR^T$, let $\pi_S \colon E \to \RRR^S$ denote the natural projection given by $\pi_S(x) = x|S$ for $x \in E$.
This is a slight abuse of notation since we do not specify the domain of $\pi_S$; however, it should not lead to confusion, as the domain will always be clear from context.
Note that for any $X \subset \KKK$, the map $i_{\Delta(X)}$ is just the projection $\pi_X$ if we identify $C(\KKK)$ with $i_{M(\KKK)}(C(\KKK))$. In particular, if $C(\KKK)$ is endowed with the pointwise topology, then $i_{\Delta(\KKK)}$ is simply the identity.\smallskip

For the next lemma, one should note the similarities in the proof of Theorem~3.7 in \cite{M2}. For the convenience of the reader, below we include the whole argument.

\begin{lem}\label{lem:ctble_coordinates}
For every $f \in C(\KKK)$ there exists a countable set $A \subset (0,1)$ such that $f \in E(A)$.
\end{lem}

\begin{proof}
Striving for a contradiction, assume that there is an uncountable set $Z \subset (0,1)$ and a function $f \in C(\KKK)$ satisfying $f(x,0) \neq f(x,1)$ for all $x \in Z$. Without loss of generality, we may assume that there exists an $\epsilon > 0$ such that 
\begin{equation}\label{eq:cond_of_Z}
    \forall x \in Z \quad |f(x,0) - f(x,1)| > \epsilon.
\end{equation}
Let $z \in Z$ be a condensation point of $Z$ in the euclidean topology and let $B_i = B(f(z,i),\tfrac{\epsilon}{2})$ denote the open unit balls centered at $f(z,i)$ of radius $\tfrac{\epsilon}{2}$ for $i=0,1$ respectively. It is clear that the sets $B_0, B_1$ are disjoint and hence the sets $U_i = f^{-1}(B_i)$ for $i = 0,1$ are disjoint as well. Moreover, for $i = 0,1$ we have $(z,i) \in U_i$. Fix a positive number $\delta > 0$ with
$$
(z,0) \in \Big( (z - \delta,z] \times \{0\} \Big) \cup \Big( (z - \delta, z) \times \{1\} \Big) \subset U_0,
$$
$$
(z,1) \in \Big( (z, z + \delta) \times \{0\} \Big) \cup \Big( [z, z + \delta) \times \{1\} \Big) \subset U_1.
$$
Since $z$ is an accumulation point of $Z$, at least one of the sets $(z-\delta,z) \cap Z, (z, z + \delta)\cap Z$ is uncountable. Without loss of generality assume that this is the case for the first set above. Then for all $x \in (z-\delta,z) \cap Z$ we have $(x,0),(x,1) \in U_0$ and hence $|f(x,0)-f(x,1)| < 2\cdot\tfrac{\epsilon}{2} = \epsilon$, which is a contradiction with (\ref{eq:cond_of_Z}).
\end{proof}

The following factorization results will remain useful (see \cite[Lemma 4.1]{M1}).

\begin{lem}\label{lem:general_factor}
Let $X$ be a set and let $E$ be a linear subspace of the product $\RRR^X$. If $f \colon E \to \RRR^{\omega}$ is a continuous function, then $f$ depends only on countably many coordinates, i.e., there is a countable set $Y \subset X$ and a continuous function $g \colon \pi_Y(E) \to \RRR^{\omega}$ satisfying:
$$
f = (g \circ \pi_Y)|E.
$$
\end{lem}

\begin{lem}\label{lem:special_factor}
For any subset $A \subset (0,1)$ and any homeomorphism $\varphi \colon E(A) \to E(A) \times E(A)$, where $E(A)$ is considered as a subspace of either $C_w(\KKK)$ or $C_p(\KKK)$, there exists a countable set $S \subset M(\KKK)$, containing $\{\delta_{q_n} \colon n \in \omega\}$, such that the function
$$
(i_S \times i_S) \circ \varphi \circ i_S^{-1}
$$
maps homeomorphically $i_S(E(A))$ onto $i_S(E(A)) \times i_S(E(A))$.
\end{lem}

\begin{proof}
First, consider the space $E(A)$ endowed with the weak topology and note that $i_{M(\KKK)} \colon E(A) \to \RRR^{M(\KKK)}$ is a homeomorphic embedding.
Let $F = i_{M(\KKK)}(E(A))$ and let $\psi \colon F \to F \times F$ be a homeomorphism given by 
$$
\psi = (i_{M(\KKK)} \times i_{M(\KKK)}) \circ \varphi \circ i_{M(\KKK)}^{-1}.
$$
By back--and--forth induction we will construct a sequence of countable sets
$$
S_0 \subset S_1 \subset \dots \subset M(\KKK)
$$
satisfying the following conditions:
\begin{itemize}
    \item $S_0 = \{\delta_{q_n} \colon n \in \omega\}$,
    \item For even $j$ the function
    $$
    (\pi_{S_j} \times \pi_{S_j}) \circ \psi \circ \pi_{S_{j + 1}}^{-1} \colon \pi_{S_{j + 1}}(F) \to \pi_{S_j}(F) \times \pi_{S_j}(F)
    $$
    is continuous,
    \item For odd $j$ the function
    $$
    \pi_{S_{j}} \circ \psi^{-1} \circ(\pi_{S_{j + 1}}^{-1} \times \pi_{S_{j + 1}}^{-1}) \colon \pi_{S_{j + 1}}(F) \times \pi_{S_{j + 1}}(F) \to \pi_{S_{j}}(F)
    $$
    is continuous.
\end{itemize}
Recall that since the set $\{q_n \colon n \in \omega\}$ is dense in $\KKK$ and the set $\{\delta_{q_n} \colon n \in \omega\}$ is promised to be in $S_j$, then the maps $\pi_{S_j}$ will be injections, and so the map $\pi_{S_j}^{-1}$ is well-defined on the image of $\pi_{S_j}$.\smallskip

Assume that for some $j$ the sets $S_0, \dots, S_j$ have been already chosen. If $j$ is even, then the function
$$
(\pi_{S_{j}} \times \pi_{S_{j}}) \circ \psi \colon F \to \pi_{S_j}(F) \times \pi_{S_j}(F)
$$
is continuous and maps the linear subspace $F \subset \RRR^{M(\KKK)}$ into $\RRR^{S_j} \times \RRR^{S_j}$. By Lemma \ref{lem:general_factor} there is a countable set $Y \subset M(\KKK)$ and a continuous function
$$
\psi_j \colon \pi_{Y}(F) \to \RRR^{S_j} \times \RRR^{S_j}
$$
satisfying $(\pi_{S_j} \times \pi_{S_j}) \circ \psi = (\psi_j \circ \pi_{Y})|F$. In particular, this means that 
$$
\psi_j \colon \pi_{Y}(F) \to \pi_{S_j}(F) \times \pi_{S_j}(F).
$$
Finally, put $S_{j + 1} = S_j \cup Y$. If $j$ is odd, then the argument is analogous.\smallskip

This finishes the inductive construction. Let $S = \bigcup_{j = 0}^{\infty} S_j$. It is straightforward to check that $(i_S \times i_S) \circ \varphi \circ i_S^{-1} \colon i_S(E(A)) \to i_S(E(A)) \times i_S(E(A))$ is a homeomorphism, since 
$$
(i_S \times i_S) \circ \varphi \circ i_S^{-1} = (\pi_S \times \pi_S) \circ \psi \circ \pi_S^{-1},
$$
and basic open neighborhoods in $i_S(E(A))$ are given by finite subsets of $S$, all of which, by construction, are included in some intermediate set $S_j$. This finishes the proof in the case of weak topology. 
To obtain the result in the case of pointwise topology, one may replace the map $i_{M(\KKK)}$ with $i_{\Delta(\KKK)}$ (then $\psi$ from the above reasoning is just $\varphi$ we have started with).
\end{proof}

The lemma below allows us to count all possible homeomorphisms, which we want to eradicate. This makes the inductive construction possible. One should compare it with \cite[Lemma 4.3]{M1}.

\begin{lem}\label{lem:counting_homeomorphisms}
There exists a family of functions $\Ff = \{\varphi_{\alpha} \colon \alpha < 2^{\omega}\}$ satisfying:
\begin{enumerate}
    \item[$(i)$] $\dom(\varphi_\alpha) \subset C(\KKK)$, $\ran(\varphi_{\alpha}) \subset C(\KKK) \times C(\KKK)$ and $\varphi_{\alpha}$ is 1--1 for each $\alpha < 2^{\omega}$,
    \item[$(ii)$] For any $A \subset (0,1)$ and any homeomorphism $\psi \colon E(A) \to E(A) \times E(A)$, where $E(A)$ is endowed with either weak or pointwise topology, there exists $\alpha < 2^{\omega}$ with $\varphi_{\alpha}|E(A) = \psi$.
\end{enumerate}
\end{lem}

\begin{proof}
Let $\Ff$ denote the family of all functions $\varphi$ with the following property: there exists a countable set $S \subset M(\KKK)$, a $G_{\delta}$ set $X \subset \RRR^S$ and a homeomorphic embedding $h \colon X \to \RRR^S \times \RRR^S$ satisfying the following conditions:
\begin{enumerate}
    \item[$(a)$] $\{\delta_{q_n} \colon n \in \omega\} \subset S$,
    \item[$(b)$] $\dom(\varphi) = i_S^{-1}\left(h^{-1}\left(i_S\left(C(\KKK)\right) \times i_S\left(C(\KKK)\right)\right)\right)$,
    \item[$(c)$] $\varphi(f) = \left(\left(i_S^{-1} \times i_S^{-1}\right) \circ h \circ i_S\right)(f)$ for $f \in \dom(\varphi)$.
\end{enumerate}
It is easy to see that $|\Ff| = 2^{\omega}$ and that the condition $(i)$ is satisfied for all $\varphi \in \Ff$. We will check that this is also the case for the condition $(ii )$.\smallskip

Let $A \subset (0,1)$ and let $\psi \colon E(A) \to E(A) \times E(A)$ be a homeomorphism. By Lemma \ref{lem:special_factor}, there exists a countable set $S \subset \KKK$ with $\{\delta_{q_n} \colon n \in \omega\} \subset S$ such that
$$
g = (\pi_S \times \pi_S) \circ \psi \circ \pi_S^{-1} \colon \pi_S(E(A)) \to \pi_S(E(A)) \times \pi_S(E(A))
$$
is a homeomorphism. By Lavrentieff's theorem (cf. \cite[Theorem 4.3.21]{E}), there is a set $X \subset \RRR^S$ of type $G_{\delta}$, containing $\pi_S(E(A))$, along with a homeomorphic embedding
$$
h \colon X \to \RRR^S \times \RRR^S
$$
extending $g$. Now, the mapping $\varphi$ defined by conditions $(b)$ and $(c)$ is an element of the family $\Ff$ and $\varphi|E(A) = \psi$, as desired.
\end{proof}

The lemma below is our killing condition (see \cite[Lemma 5.1]{M1}). In the sequel the maps $\pi_i$, defined for $i = 1,2$, send an element of $E(A) \times E(A)$ to its $i$--th coordinate.

\begin{lem}\label{lem:killing_criterion}
Let $A \subset (0,1)$ and let $\varphi \colon E(A) \to E(A) \times E(A)$ be a homeomorphism, where $E(A)$ is endowed with either weak or pointwise topology. Then, for every $x \in A$ one of the following conditions holds:
\begin{enumerate}
    \item There exists $f \in E(A)$ such that $f(x,0) = f(x,1)$ and 
    $$
    (\pi_i \circ \varphi(f))(x,0) \neq (\pi_i \circ \varphi(f))(x,1)
    $$
    for some $i \in \{1,2\}$.
    \item There exists $f \in E(A)$ such that $f(x,0) \neq f(x,1)$ and 
    $$
    (\pi_i \circ \varphi(f))(x,0) = (\pi_i \circ \varphi(f))(x,1)
    $$
    for all $i \in \{1,2\}$.
\end{enumerate}
\end{lem}

\begin{proof}
Fix $x \in A$ and let $U = \{f \in C(\KKK) \colon f(x,0) \neq f(x,1)\}$ and 
$$
V = (U \times E(A)) \cup (E(A) \times U).
$$
If neither condition $(1)$ nor $(2)$ holds, then $\varphi(U) = V$, which is not possible, since $V$ is connected and $U$ is not.
\end{proof}

Now, we are fully equipped to prove the main result of this section.

\begin{thm}\label{thm:main}
There exists a set $X \subset (0,1)$ such that the space $C(I_X)$ is not homeomorphic to $C(I_X) \times C(I_X)$ with regard to both the weak and the pointwise topologies.
\end{thm}

\begin{proof}
Let $\Ff = \{\varphi_{\alpha} \colon \alpha < 2^{\omega}\}$ be a family of functions from Lemma \ref{lem:counting_homeomorphisms}. By induction, we will pick sets $X_{\alpha} \subset (0,1)$ and points $z_{\alpha} \in (0,1)$ for $\alpha < 2^{\omega}$, which will satisfy the following conditions for every $\alpha < 2^{\omega}$:
\begin{enumerate}
    \item $1 \leq |X_{\alpha}| \leq \omega$
    \item $X_{\beta} \cap X_{\alpha} = \emptyset$ for each $\beta < \alpha$
    \item $z_{\beta} \notin X_{\gamma}$ for each $\beta,\gamma \leq \alpha$
    \item if there exists a set $A \subset (0,1)$ with:
    \begin{itemize}
        \item $|A| = 2^{\omega}$
        \item $z_{\beta} \notin A$ for $\beta < \alpha$
        \item $\varphi_\alpha|E(A)$ a homeomorphism between $E(A)$ and $E(A) \times E(A)$
    \end{itemize}
    then one of the following conditions must be satisfied:
    \begin{enumerate}
        \item[($a$)] there exists $f \in E(\bigcup_{\beta \leq \alpha} X_{\beta})$ satisfying
        $$
        (\pi_i \circ \varphi_{\alpha}(f))(z_{\alpha},0) \neq (\pi_i \circ \varphi_{\alpha}(f))(z_{\alpha},1)
        $$
        for some $i = 1,2$;
        \item[($b$)] there exist $f,g \in E(\bigcup_{\beta \leq \alpha} X_{\beta})$ such that
        $$
        \varphi_{\alpha}^{-1}(f,g)(z_{\alpha},0) \neq \varphi_{\alpha}^{-1}(f,g)(z_{\alpha},1).
        $$
    \end{enumerate}
\end{enumerate}
Assume that the sets $X_{\beta}$ and points $z_{\beta}$ have been picked for all $\beta < \alpha$ for some $\alpha < 2^{\omega}$. We will show how to choose $X_{\alpha}$ and $z_{\alpha}$. First, assume that a set $A \subset (0,1)$ such as in the condition $(4)$ exists. Start with picking any point $z_\alpha~\in~A~\setminus~\bigcup_{\beta < \alpha}X_{\beta}$. This is possible since $|\bigcup_{\beta < \alpha} X_{\beta}| < 2^{\omega} = |A|$.\smallskip

By Lemma \ref{lem:killing_criterion}, applied to $A$, $z_{\alpha}$ and $\varphi_{\alpha}|E(A)$, there exists $f \in E(A)$ such that either
\begin{enumerate}
    \item $f(z_{\alpha},0) = f(z_{\alpha},1)$ and there is $i \in \{1,2\}$ with
    $$
    (\pi_i \circ \varphi_{\alpha})(f)(z_{\alpha},0) \neq (\pi_i \circ \varphi_{\alpha})(f)(z_{\alpha},1),
    $$
    or
    \item $f(z_{\alpha},0) \neq f(z_{\alpha},1)$ and
    $$
    (\pi_i \circ \varphi_{\alpha})(f)(z_{\alpha},0) = (\pi_i \circ \varphi_{\alpha})(f)(z_{\alpha},1)
    $$
    for both $i \in \{1,2\}$.
\end{enumerate}
If the first case holds, then by Lemma \ref{lem:ctble_coordinates} there is a countable set $Y \subset A$ with $f \in E(Y)$. Since $f(z_{\alpha},0) = f(z_{\alpha},1)$, we may assume that $z_{\alpha} \notin Y$. If $Y \setminus \bigcup_{\beta < \alpha} X_{\beta} \neq \emptyset$, then let $X_{\alpha} = Y \setminus \bigcup_{\beta < \alpha} X_{\beta}$; otherwise pick an arbitrary $y \in A \setminus (\bigcup_{\beta < \alpha} X_{\beta} \cup \{z_{\alpha}\})$ and put $X_{\alpha} = \{y\}$ to ensure we add at least one new element.\smallskip

If the second case holds, again by Lemma \ref{lem:ctble_coordinates}, find a countable $Y \subset A$ satisfying
$$
(\pi_i \circ \varphi_{\alpha})(f) \in E(Y) \;\; \text{for} \;\; i \in \{1,2\} \;\; \text{and} \;\; z_{\alpha} \notin Y.
$$
Now, the set $X_{\alpha}$ is defined the same way as in the first case.\smallskip

If no set $A$ satisfying condition $(4)$ exists, pick arbitrary $z_{\alpha}$ and $X_{\alpha}$ satisfying conditions $(1)$--$(3)$. This finishes the inductive step of the construction.\smallskip

Let $X = \bigcup_{\alpha < 2^{\omega}} X_{\alpha}$. Note that by conditions $(1)$ and $(2)$, we have $|X| = 2^{\omega}$. Below we will check that the space $C(I_X)$ is not homeomorphic to its square, if regarded both with the weak and pointwise topologies. Striving for a contradiction, assume that there exists a homeomorphism $\psi \colon E(X) \to E(X) \times E(X)$. Note that here we freely use the identification $E(X) \sim C(I_X)$. By Lemma \ref{lem:counting_homeomorphisms} there exists $\alpha < 2^{\omega}$ with $\psi = \varphi_{\alpha}|E(X)$. Then by the condition $(4)$ from our construction, either the condition $(a)$ or $(b)$ would be satisfied.\smallskip

If condition $(a)$ holds, then there is $f \in E(\bigcup_{\beta < \alpha} X_{\beta}) \subset E(X)$ with 
\begin{align*}
(\pi_i \circ \psi)&(f)(z_{\alpha},0) = (\pi_i \circ \varphi_{\alpha})(f)(z_{\alpha},0) \neq\\ 
&\neq (\pi_i \circ \varphi_{\alpha})(f)(z_{\alpha},1) = (\pi_i \circ \psi)(f)(z_{\alpha},1)
\end{align*}
for some $i \in \{1,2\}$. Since $z_{\alpha} \notin X$,  $(\pi_i \circ \psi)(f) \notin E(X)$. If condition $(b)$ holds, then
there exist $f,g \in E(\bigcup_{\beta \leq \alpha} X_{\beta}) \subset E(X)$ such that
\begin{align*}
\psi^{-1}(f,g)(z_{\alpha},0) &= \varphi_{\alpha}^{-1}(f,g)(z_{\alpha},0) \neq\\ 
&\neq \varphi_{\alpha}^{-1}(f,g)(z_{\alpha},1) = \psi^{-1}(f,g)(z_{\alpha},1)
\end{align*}
and so $\psi^{-1}(f,g) \notin E(X)$. In both cases this is a contradiction, and so we are done.
\end{proof}

\begin{cor}
There exists a separable compact line $K$ of weight $2^{\omega}$ such that the Banach space $C(K)$ is not isomorphic to $C(K) \times C(K)$.
\end{cor}

\section{Acknowledgements.}
I am indebted to Witold Marciszewski and Mikołaj Krupski for their constant support and numerous fruitful discussions on the subject at hand.

\bibliographystyle{siam}
\bibliography{bib.bib}

\end{document}